\documentclass[11pt,a4,reqno]{amsart}

\usepackage{amsmath, amssymb}
\usepackage{graphicx}
\usepackage{subfigure}
 \usepackage{mathrsfs}
 \usepackage{bbm}
\usepackage{color}
\usepackage{esint}
\usepackage[svgnames]{xcolor} 
\usepackage[colorlinks,citecolor=red,pagebackref,hypertexnames=false,breaklinks]{hyperref}
\usepackage{pgf,tikz}
\usepackage{pdfsync}
\usepackage[bottom]{footmisc}
\usepackage{dsfont}

\usepackage[T1]{fontenc}
\usepackage{lmodern}
\usepackage{mathtools}  
\usepackage{amssymb}
\usepackage{lipsum}
\usepackage{mathrsfs}
\usepackage{color}
\usepackage{bbm}

\date{\today}
\hypersetup{
	colorlinks=true,
	linkcolor=blue,
	urlcolor=magenta,
	citecolor=red,
}

\newtheorem{lemma}{Lemma}[section]
\newtheorem{remark}{Remark}[section]

\newtheorem{corollary}{Corollary}[section]
\newtheorem{proposition}{Proposition}[section]
\newtheorem{theorem}{Theorem}[section]
\newtheorem{assumption}{Assumption}[section]
  
\def\proof{\noindent {\bf Proof}\ } 
\newcommand{\boite}{\mbox{\rule{2mm}{2mm}}}
\def\endproof{\mbox{} \hfill \boite}

\newcommand{\norm}[1]{\|#1\|}
\newcommand{\abs}[1]{\left|#1\right|}

\usepackage{algorithm}
\usepackage{algorithmicx}
\usepackage{algpseudocode}

\def\no{\nonumber}

\def\R{\mathbb{R}}

\def\im{\textrm{Im}}
\def\re{\textrm{Re}}

\def\bar{\overline}
\def\hat{\widehat}
\def\tilde{\widetilde}

\def\ep{\varepsilon}
\def\e{\^{e}}

\newcommand{\be}{\begin{equation}}
\newcommand{\ee}{\end{equation}}
\newcommand{\ba}{\begin{array}}
\newcommand{\ea}{\end{array}}

\newcommand{\bea}{\begin{eqnarray*}}
\newcommand{\eea}{\end{eqnarray*}}
\newcommand{\bean}{\begin{eqnarray}}
\newcommand{\eean}{\end{eqnarray}}

\def\e{\varepsilon}

\def\1e{\mathds{1}}

\def\Ce{\mathds{C}}

\def\Ne{\mathds{N}}

\def\Re{\mathds{R}}

\def\vg0{\mathbf{0}}

\def\P{\mathscr{P}}

\def\R{\mathscr{R}}

\def\convol#1#2{#1\ast#2}

\def\ud{\,\mathrm{d}}

\def\eqd{:=}

\def\segcc#1#2{[#1, #2]}

\def\segoo#1#2{(#1, #2)}

\def\vphan{\vphantom{\bigl|}}

\def\scal#1#2{\langle #1,#2 \rangle}

\def\set#1#2{\{\mskip 1mu #1 \mskip 1mu
    | \mskip 1mu #2 \mskip 1mu \}
    }
\def\setc#1#2{
    \left\{
    \mskip 2mu #1 \mskip 2mu
    \left| \vphan\vphantom{#1#2} \right.
    \mskip 2mu #2 \mskip 2mu
    \right\}
    }

\def\mod#1{|\mskip 1mu #1 \mskip 1mu|}
\def\modc#1{
    \left|
    \mskip 2mu #1 \vphan \mskip 2mu
    \right|
    }

\def\norm#1{\| \mskip 1mu #1 \mskip 1mu \|}
\def\normc#1{
    \left\|
    \mskip 2mu #1 \vphan \mskip 2mu
    \right\|
    }

\def\operator#1#2#3#4#5{
        \begin{array}{lcll}
        \displaystyle #1 \colon & \displaystyle #2 
        &\longrightarrow & \displaystyle #3 \\[.5ex]
                     & \displaystyle #4 & \longmapsto     & \displaystyle #5
        \end{array}
        }

\def\argmin#1{\mathop{\mathrm{argmin}}#1}

\def\ran#1{\mathop{\mathrm{ran}}#1}

\def\re{\mathop{\mathrm{Re}}}
\def\im{\mathop{\mathrm{Im}}}

\def\ch{\mathop{\mathrm{ch}}}

\begin{document}

\title{Regularization of the inverse Laplace transform by Mollification}

\author{Pierre Maréchal}

\address{Pierre Maréchal, Institut de Mathématiques de Toulouse, UMR CNRS 5219,
Université de Toulouse 3 Paul Sabatier, 118 route de Narbonne, 31062 Toulouse cedex 9, France}

\email{pierre.marechal@math.univ-toulouse.fr}
\author{Faouzi Triki}

\address{Faouzi Triki,  Laboratoire Jean Kuntzmann,  UMR CNRS 5224, 
Universit\'e  Grenoble-Alpes, 700 Avenue Centrale,
38401 Saint-Martin-d'H\`eres, France}

\email{faouzi.triki@univ.grenoble-alpes.fr}

\author{Walter C. Simo Tao Lee}
\address{Walter C. Simo Tao Lee, Universit\'e Toulouse Capitole, 2 Rue du Doyen Gabriel Marty, 31000 Toulouse, France}
\email{wsimotao@gmail.com}

\thanks{
The work is supported in part by the
grant ANR-17-CE40-0029 of the French National Research Agency ANR (project MultiOnde).}


\maketitle

\begin{abstract}
In this paper we study  the inverse Laplace transform. 
We first derive a new global  logarithmic stability estimate that shows
that the inversion is severely ill-posed. Then we propose a regularization
method to compute the inverse Laplace transform  using the concept of mollification.
Taking into account the exponential instability we derive a criterion for 
selection of the regularization parameter. We show that
by taking the optimal value of this parameter we improve 
significantly the convergence of the method. Finally, making use of
the holomorphic extension of the Laplace transform, we suggest a new PDEs based
numerical method for the computation of the solution. The effectiveness of the
proposed regularization method is demonstrated through several numerical examples.  
\end{abstract}

\section{Introduction}

In many problems dealing with time evolution PDEs where
the coefficients are time independent, the Laplace transform
is other used to study the existence, and uniqueness and
regularity of the solution, or to compute it numerically/analytically,
and so knowing the stability of the Laplace inverse transform is of
importance (see for example \cite{kian2021logarithmic}). While it is well understood from the numerical
approximation point of view  that this inversion is exponentially
ill-posed, only partial  stability estimates
have been derived \cite{epstein2008bad,lederman2016stability}.

A large number of numerical inversion methods have been 
developed during the last decades for solving Laplace inversion.
Their efficiency and convergence heavily depend on the choice
of some intrinsic parameters. In what follows we briefly
review some of the existing methods. Numerical inversion
of the Laplace transform using Fourier series approximations
were first applied by Dubner and Abate in the late
sixties~\cite{dubner1968numerical}. Other authors applied
different techniques to speed up the convergence of the Fourier series 
\cite{durbin1974numerical, davies1979numerical, mcwhirter1978numerical}. 
The second type of numerical inversion is based on collocation methods.
The third type is inspired by Talbot's idea that consists in
deforming the Bromwich contour into a curve that
allows for a better numerical integration \cite{talbot1979accurate}.
These methods achieved a notable progress in computing the Laplace
transform inverse but the choice of some involved parameters was
somewhat arbitrary. 
Finally, the approximate solutions  obtained by all these numerical methods may differ significantly from the targeted ones when the data are noisy.

In this paper we first derive a global stability estimate for the Laplace inversion in Theorem \ref{stability estimate}. We also propose a regularization method  to deal with  the inherent  exponential using the concept of mollification.
Indeed we considered solving the minimization problem $(\P)$ given in \eqref{minimization problem}. Under smoothness condition we derived an order-optimal convergence rate for data without and with noise  in respectively Theorem \ref{Theorem order optim 1 under noisy data only} and Corollary \ref{Theorem order optim 1 under noisy data only}. Using the fact that the Laplace transform has a holomorphic extension in the right half-plane, we propose a PDEs based numerical method to solve the 
inverse problem. Precisely we show that computing the inverse Laplace transform is equivalent to solving a Cauchy problem for the Laplace equation  in quarter of the plane \eqref{cauchy problem u}. We finally solve the regularized Cauchy problem \eqref{def u beta num sec} by applying 
a first order optimality condition and  using the fourth order finite difference scheme method \eqref{equation U beta}. \\

 The paper is organized as follows. 
In the second section we introduce the appropriate harmonic analysis to study the forward and inverse problems. Precisely
we show that the Laplace inverse is an invertible map from the
set of square integrable function onto the Hardy space.
Then we consider the problem of recovering a function
from the knowledge of its Laplace transform only on the real axis. We finally study the ill-posedness of the inversion by applying unique continuation techniques for holomorphic functions. 

In Section~3 we propose a regularization method to
compute the inverse Laplace transform. Based on the stability
estimates found in the second section we derive a new
criterion for selection of the regularization parameter.
We show that taking the obtained optimal value of this
parameter improves significantly the  convergence.

Finally, we provide several numerical examples to validate 
the effectiveness of the proposed regularization method in Section~4. 

\section{Stability estimates}

Let $F(f)$ or $\hat{f}$ denote the
Fourier transform of a function $f\in L^1(\Re)$, defined as
$$
Ff(\xi)=\widehat{f}(\xi)=\frac{1}{\sqrt{2 \pi}}\int_{\Re} e^{-i\xi t} f(t) \ud t.
$$
Fix $c\in \Re_+$, and let $\Pi_c$ be the right half-plane
$$
\Pi_c=\setc{z\in\Ce}{\re(z)>c}.
$$
We consider the Hardy space $H^2(\Pi_c)$, defined as the space
of holomorphic functions~$h$ in $\Pi_c$ for which
$$
\norm{h}_{\Pi_c}^2\eqd
\sup_{s>c}\int_{-\infty}^{\infty} |h(s+it)|^2\ud t <\infty.
$$
For $q\geq 1$, define 
$$
L^q_c\segoo{0}{\infty}=
\setc{f}{e^{-ct}f(t)\in L^{q}\segoo{0}{\infty}}.
$$
Let $W^{1,1}_{c}\segoo{0}{\infty}=
\setc{f}{e^{-ct}f(t)\in W^{1,1}\segoo{0}{\infty}}$
the weighted Sobolev space endowed with the norm
$$
\norm{f}_{W^{1,1}_c\segoo{0}{\infty}} =
\norm{f}_{L^1_c\segoo{0}{\infty}}+\|f^\prime\|_{L^1_c\segoo{0}{\infty}}.
$$
Set $W^{1,1}_{c, 0}\segoo{0}{\infty}$ to be the
closure of $C^\infty_0\segoo{0}{\infty}$ in
$W^{1,1}_{c}\segoo{0}{\infty}$.
For $f\in L^2_c\segoo{0}{\infty}$, we define its Laplace transform by
\begin{equation}
\label{Laplace}
L f(z)= \int_0^{\infty} e^{-z\tau} f(\tau)\ud\tau.
\end{equation}

\begin{proposition}\sf
\label{prop1}
The Laplace transform
$L\colon L^2_c\segoo{0}{\infty}\to H^2 (\Pi_c)$
is an invertible bounded operator.  In addition, we have 
\bean
\label{identity}
\frac{1}{2\pi}\int_{-\infty}^{\infty} |L f(s+it)|^2 \ud t=
\int_0^{\infty} e^{-2s\tau}|f(\tau)|^2 \ud\tau,
\eean
for all $s\geq c$, which implies that
$\|L f \|_{\Pi_c }= \|f \|_{L^2_c\segoo{0}{\infty}}$.
\end{proposition}

\proof
We first observe that $L f(z)$  has a holomorphic extension
to the right half-plane $\Pi_c$, and we have  
\bea
L f (s+it)= 
\int_0^{\infty} e^{-i t \tau} e^{-s\tau} f(\tau) \ud\tau,
\quad t\in\Re.
\eea
We also remark that $t\to Lf (s+it)$ is the
Fourier transform of $e^{-s\tau} f(\tau)\chi_{\segoo{0}{\infty}}(\tau)$
which lies in $L^2\segoo{0}{\infty}$ for $s=c$,
and in $L^2\segoo{0}{\infty}\cap L^1\segoo{0}{\infty}$ for $s>c$.
Applying the classical inverse Fourier transform, we get 
\bea
e^{-s\tau} f(\tau)=
\frac{1}{2\pi}\int_{-\infty}^{\infty} e^{i\tau t} L  f(s+it)\ud t,
\quad t>0.
\eea
The equality \eqref{identity} is then a direct consequence of
the Parseval identity. Moreover one can easily check that 
\[
\sup_{s>c} \int_{-\infty}^{\infty} |L  f (s+it)|^2 \ud t=
\|f\|_{L^2_c\segoo{0}{\infty}},
\]
which implies that the Laplace transform operator~$L$ is a
bounded operator from $L^2_c\segoo{0}{\infty}$ into $H^2(\Pi_c)$.
We also deduce from identity~\eqref{identity} that
$L$ is injective. Finally, the Paley-Wiener Theorem shows that the
Laplace transform is surjective from $L^2_c\segoo{0}{\infty}$  
onto $H^2(\Pi_c)$ (Theorem 19.2 page 372 in \cite{rudin1974real}).
\endproof

In practice the Laplace transform $Lf$ is known only on $\Re_+$,
and the Laplace inversion encountered in applications
consists in recovering a function $f\in L^1_c\segoo{0}{\infty}$
from the knowledge of its real-valued Laplace transform
$Lf(z)$ for $z\in\segoo{c}{\infty}$.
Therefore, we consider the real-valued Laplace operator~$L$
from $L^1_c\segoo{0}{\infty}$ to $L^\infty\segoo{c}{\infty}$. \\

 Let $f \in W^{1,1}_{c}\segoo{0}{\infty}$, and set $g = Lf$. We are interested in this paper in  {\sf  the  inverse problem of recovering
$f$ from the knowledge of  $g(z), \; z\in (c, \infty)$. }
 \begin{proposition}\sf
 \label{operatorL}
The operator
$L\colon L^1_c\segoo{0}{\infty}\to L^\infty\segoo{c}{\infty}$
is bounded and one to one. 
\end{proposition}

\proof 
We deduce from Proposition \ref{prop1} that
$L$ is a bounded operator. Moreover, 
$$
\norm{Lf}_{L^\infty\segoo{c}{\infty}}\leq
\norm{f}_{L^1_c\segoo{0}{\infty}}.
$$
Since $Lf$ has a unique holomorphic extension
to $\Pi_c$, $Lf=0$ on $\{0\}\times\segoo{c}{\infty}$
implies immediately that $Lf=0$ on $\Pi_c$.
We deduce from Proposition~\ref{prop1} that $f=0$.
\endproof



We further study the ill-posedness of the inverse problem
by applying unique continuation techniques
for holomorphic functions. 
Next, we present the main result of this section.

\begin{theorem} \label{stability estimate}\sf
Let $f\in W^{1, 1}_{0, c}\segoo{0}{\infty}$, and
$ M_f = \|f\|_{W^{1,1}_c\segoo{0}{\infty}}$, and set
$$
\ep=\|L  f\|_{L^\infty\segoo{c}{\infty}}.
$$ 
Then,
\bean
\label{mainestimate}
\norm{f}_{L^2_{c+1}\segoo{0}{\infty}} \leq 2(3+c)
M_f\frac{1}{\abs{\ln(\frac{\ep}{M_f})}^{\frac{1}{4}} }.
\eean
\end{theorem}

\proof
Since $W^{1,1}_{0, c}\segoo{0}{\infty} \subset
L^p\segoo{0}{\infty}$ for all $p\geq 1$,
$L  f(z)\in H^2 (\Pi_c)$ is a holomorphic
function on~$\Pi_c$, and satisfies in addition 
\bean
\label{ffBound}
\left| L  f(z)\right|\leq
\|f\|_{L^1_c\segoo{0}{\infty}}\leq
\|f\|_{W^{1,1}_c(]0, \infty[)}=
M_f, \quad \forall z\in \Pi_c.
\eean

We deduce from Theorem \ref{thmUC}
with  $F(z)=L  f(z)\in H^2(\Pi_c)$ that 
\bean
\label{origin}
|L  f(z)(c+1+it)|\leq 
M_f \left(\frac{\ep}{M_f}\right)^{w^\pm(c+1+it)}, \quad 
\forall t\in\Re_\pm.
\eean

Since $0<\frac{\ep}{M_f} \leq 1$, we obtain 
\bean
\int_0^R\abs{F(c+1+ip)}^2 \ud p
&\leq
\int_0^R\abs{M_f\left(\frac{\ep}{M_f}\right)^{w_+(c+1+ip)}}^2\ud p\leq
R M^2_f \left(\frac{\ep}{M_f}\right)^{2\underset{0<p<R}{\min} w_+(c+1+ip) }\no \\
&\leq
R M^2 \left(\frac{\ep}{M_f}\right)^{2(1-\frac{2}{\pi}\arctan(R))}
\label{FB}.
\eean
We have by analogy
\bean
\int^0_{-R}\abs{F(c+1+ip)}^2 \ud p 
&\leq
\int^0_{-R} \abs{M_f ^{w_-(c+1+ip)} }^2\ud p\leq
R M^2_f\left(\frac{\ep}{M_f}\right)^{2\underset{0<p<R}{\min} 
w_-(c+1+ip)} \no \\
&\leq
R M^2_f  \left(\frac{\ep}{M_f}\right)^{2(1-\frac{2}{\pi} arctan(R))}
\label{FB2}.
\eean
We deduce from  inequalities (\ref{FB}) and (\ref{FB2}) that 
\bean 
\int_{-R}^R \abs{F(c+1+ip)}^2 \ud p \leq 2 R M^2_f  \left(\frac{\ep}{M_f}\right)^{2(1-\frac{2}{\pi} \arctan(R))}.
\label{*}
\eean
Recall that $f(0)=0$. Therefore, we have 
\bean
|p| |F (c+1+ip)|
&=
\left|\int_0^{\infty} p e^{-ipr} f(r)e^{-(1+c)r}dr \right|=
\abs{\int_0^{\infty} e^{-ipr}(f(r)e^{-(1+c)r})^\prime \ud r}\nonumber \\
&\leq
(2+|c|)M_f
\label{**}.
\eean
Using inequalities (\ref{*}) and (\ref{**}) gives
\bean
\int_\R \abs{F(c+1+ip)}^2 \ud p
&=
\int^R_{-R} \abs{F(c+1+ip)}^2 \ud p+
\int_{\abs{p}>R} \abs{F(c+1+ip)}^2 \ud p \no \\
&\leq
2 M^2_f  R\left(\frac{\ep}{M_f}\right)^{2(1-\frac{2}{\pi} \arctan(R))} +2(2+|c|)^2M_f^2  \dfrac{1}{R}.
\label{c1}
\eean

We fix $R\geq 1$. On the one hand, we have
$1-\dfrac{1}{3R^2}\geq \dfrac{1}{2}$, then 
$$
\frac{1}{R}\left(1-\frac{1}{3R^2}\right)\geq \frac{1}{2R}.
$$
On the other hand, we have
\bea
\frac{\pi}{2}-\arctan(R)= \arctan\left(\frac{1}{R}\right) 
\geq \frac{1}{R} \left(1-\frac{1}{3R^2}\right).
\eea
Then, for $R\geq 1$, we have 
\bean\label{c11}
\frac{\pi}{2} -\arctan(R) \geq \frac{1}{2R}.
\eean
Combining Inequality (\ref{c11}) with the estimate (\ref{c1}) yields 
\bea
\int_\R \abs{F(c+1+ip)}^2 \ud p \leq
2 M^2_f  R\left(\frac{\ep}{M_f}\right)^{\frac{2}{\pi R}} +2(2+|c|)^2M_f^2 \dfrac{1}{R},
\quad\forall R\geq 1.
\eea
Now, by taking $R \left(\frac{\ep}{M_f}\right)^{\frac{2}{\pi R}}=\frac{1}{R}$,
we obtain 
\bean\label{in1}
\int_\R \abs{F(c+1+ip)}^2 \ud p\leq
2(3+|c|)^2 \frac{M_f^2}{R},
\quad R \geq 1, 
\eean
and 
\bea
-\frac{1}{\pi}\ln\left(\frac{\ep}{M_f}\right)=R\ln(R).
\eea
We remark that $\dfrac{\ln(R)}{R}\leq e^{-1}$, for $R>e$.

Therefore we have
$-\dfrac{1}{\pi}\ln\left(\frac{\ep}{M_f}\right)=R \ln(R)\leq R^2e^{-1}$.
In other words 
\bean\label{in2}
\frac{1}{R}\leq\sqrt{\frac{\pi}{e}}
\frac{1}{\sqrt{\abs{\ln\left(\frac{\ep}{M_f}\right)}}}, \quad R>e.
\eean
Substituting (\ref{in2}) in the estimate (\ref{in1}), we obtain 
\bean \label{eqqq2}
\int_\R \abs{F(c+1+ip)}^2 \ud p\leq
2\sqrt{\frac{\pi}{e}}(3+|c|)^2M_f^2
\frac{1}{\sqrt{\abs{\ln\left(\frac{\ep}{M_f}\right)}}}.
\eean
We deduce from Inequality \eqref{eqqq2} that
\bea
\int_\R \abs{F(c+1+ip)}^2 \ud p\leq
\phi^2\left(\frac{\ep}{M_f}\right),
\eea
where
\bean
\label{phi}
\phi(r)=
2(3+c)M_f\frac{1}{\abs{\ln\left(r\right)}^{\frac{1}{2}}}, \;\; \textrm{ for all }r>0.
\eean
By the Fourier Plancherel Theorem,  we have 
\bea
\int_0^{\infty}\abs{f(t)}^2 e^{-2(1+c)t}\ud r=
\dfrac{1}{2\pi}\int_\R\abs{F(c+1+ip)}^2\ud p\leq 
\phi^2(\ep).
\eea
Hence 
\bea
\norm{f}^2_{L^2_{c+1}\segoo{0}{\infty}}\leq
\phi^2\left(\frac{\ep}{M_f}\right),
\eea
which completes the proof of the theorem. 
\endproof

\section{Regularization}
Since the considered inverse problem is ill-posed, a regularization method is needed
to obtain a stable approximate inversion of the real Laplace operator. 
In this section we propose a regularization method to compute the Laplace
transform inverse, using the concept of mollification.
Based on the stability estimate found in the previous
section we derive a new criterion for selection of the regularization parameter.
We show that taking the obtained optimal value of this parameter 
improve significantly the  convergence. 

We next assume that $c=0$, and to ease the analysis we consider  the Laplace operator from 
$L^2\segoo{0}{\infty}$ onto itself. In this section and
the next one, $\norm{\cdot}$ always denotes the standard $L^2$-norm
either on $(0,\infty)$ or on~$\Re$ depending on the domain of the argument.

\begin{proposition}{\cite{hardy1933constants,setterqvist2005unitary,boumenir1998inverse}}
\label{operatorLL}\sf
The operator $L\colon L^2\segoo{0}{\infty}\to L^2\segoo{0}{\infty}$
is a bounded self-adjoint operator with purely absolutely continuous
spectrum, of multiplicity one, lying in the interval $\segcc{-\pi}{\pi}$.
Moreover, $\norm{L}=\sqrt{\pi}$. 
\end{proposition}

Let $V\colon L^2\segoo{0}{\infty}\to L^2\segoo{-\infty}{\infty}$
be the unitary operator defined by
$$
(Vf)(u)=e^{u/2}f(e^u).
$$

\begin{theorem}
{\cite{setterqvist2005unitary,boumenir1998inverse}}
\label{L2-V*KV}\sf
The operator $L\colon L^2\segoo{0}{\infty}\to L^2 \segoo{0}{\infty}$
is a bounded selfadjoint injective operator, and
$$
L^2=L^*L=V^*KV,
$$
in which~$K$ is the convolution operator of kernel
$$
k(x)=\frac{1}{2\ch(x/2)}.
$$
\end{theorem}

Note that the Fourier transform of~$k$ is given by
$$
\hat{k}(\xi)=\frac{\sqrt{\pi/2}}{\ch(\pi\xi)}.
$$
We propose here to regularize the ill-posed equation $Lf=g$ via
the variational form of
mollification~\cite{alibaud2009variational,bonnefond2009variational}.
Mollification consists in aiming at the reconstruction
of a smoothed version of the unknown function $f$.
Let $\varphi_\beta$ be an approximate unity of the form:
$$
\varphi_\beta(t)=\frac{1}{\beta}\varphi\left(\frac{t}{\beta}\right),
$$
in which $\varphi$ is some integrable function with unit integral
and other desirable properties, such as positivity, smoothness, parity.
Notice that $\hat\varphi(0)=1/\sqrt{2\pi}$.
The following mild assumption will be in force:

\begin{assumption}\sf
\label{assumption-phi-chapeau}
For every $\xi\in \mathbb R\setminus\{0\}$, $\widehat\varphi(\xi)< 1/\sqrt{2\pi}$.
\end{assumption}

Notice that this assumption is satisfied in particular if $\varphi$ is
even and nonnegative.
In the sequel, the Fourier-Plancherel operator on~$L^2\segoo{-\infty}{\infty}$
is denoted by~$F$ and, given an $L^\infty$-function~$h$, the operator
of multiplication by~$h$ is denoted by~$[h]$. We denote
by $C_\beta$ the operator of convolution by $\varphi_\beta$:
$$
\operator{C_\beta}{L^2\segoo{-\infty}{\infty}}
{L^2\segoo{-\infty}{\infty}}{f}{C_\beta f:=\convol{\varphi_\beta}{f}.}
$$
Further $g$ is a fixed function in $L^2\segoo{-\infty}{\infty}$
(not necessarily in the range of $L$). We define our
regularized solution via the optimization problem
\begin{eqnarray} \label{minimization problem}
(\P)\left|
\begin{array}{rl}
\hbox{Min}
    &\norm{g-Lf}^2+\norm{(I-C_\beta)Vf}^2\\[1ex]
\hbox{s.t.}
    &f\in  L^2\segoo{0}{\infty} 
\end{array}
\right.
\end{eqnarray}

Here, $I$ denotes the identity operator on $L^2\segoo{-\infty}{\infty}$
and $\norm{\cdot}$ denotes the standard $L^2$-norm.

For fixed $\beta>0$,
the operators $L$ and $Q_\beta\eqd (I-C_\beta)V$
satisfy the Morozov completion condition:
\begin{equation}
\label{morozov}
\forall f\in L^2\segoo{0}{\infty},\quad
\norm{Lf}^2+\norm{Q_\beta f}^2\geq \theta_\beta\norm{f}^2,
\end{equation}
in which $\theta_\beta$ is a positive constant.
As a matter of fact, using Theorem~\ref{L2-V*KV} and the unitarity
of~$F$ and~$V$, we have:
$$
\normc{Lf}^2+\norm{Q_\beta f}^2=
\normc{\left[\big(\sqrt{2\pi}\,\hat{k}+\mod{1-\sqrt{2\pi}\,\hat\varphi_\beta}^2\big)^{1/2}\right]FVf}^2\geq
\theta_\beta\norm{FVf}^2,
$$
where $\theta_\beta\eqd\inf\big(\sqrt{2\pi}\,\hat{k}+\mod{1-\sqrt{2\pi}\,\hat\varphi_\beta}^2\big)>0$
from the Riemann-Lebesgue lemma and Assumption~\ref{assumption-phi-chapeau}.
In turn, the Morozov completion condition~\eqref{morozov} implies that
the unique solution to the Problem~$(\P)$, namely
$$
f_\beta=\big(L^*L+Q_\beta^*Q_\beta)^{-1}L^*g,
$$
depends continuously on~$g\in L^2\segoo{0}{\infty}$.
Moreover, the following result holds.

\begin{theorem}{\label{Theo consistency}}\sf
In the above setting, $f_\beta$ converges strongly 
to~$f^\dagger\eqd L^\dagger g$ in $L^2\segoo{0}{\infty}$, as $\beta\downarrow 0$.
\end{theorem}

\proof
We shall prove that, for every
sequence~$(\beta_n)$ which converges to zero,
the sequence $(f_{\beta_n})$
converges strongly to~$f^\dagger$.
By assumption, $g=Lf^\dagger+g^\perp$, in which $f^\dagger\in L^2\segoo{0}{\infty}$
and $g^\perp\in(\ran{L})^\perp=\ker{L^*}$. We have:
\begin{eqnarray*}
\normc{f_\beta}
&=&
\normc{(L^*L+Q_\beta^*Q_\beta)^{-1}L^*(Lf^\dagger+g^\perp)}\\
&=&
\normc{(L^*L+Q_\beta^*Q_\beta)^{-1}L^*Lf^\dagger}\\
&\leq&
\normc{f^\dagger}.
\end{eqnarray*}
As a matter of fact, letting $U\eqd FV$, for every~$f$,
$$
(L^*L+Q_\beta^*Q_\beta)^{-1}L^*Lf=
U^*\left[\frac{\sqrt{2\pi}\,\hat{k}}{\sqrt{2\pi}\,\hat{k}+\mod{1-\sqrt{2\pi}\,\hat\varphi_\beta}^2}\right]Uf,
$$
which implies that
$$
\normc{(L^*L+Q_\beta^*Q_\beta)^{-1}L^*Lf}^2=
\normc{U^*\left[\frac{\sqrt{2\pi}\,\hat{k}}{\sqrt{2\pi}\,\hat{k}+\mod{1-\sqrt{2\pi}\,\hat\varphi_\beta}^2}\right]Uf}^2\leq
\normc{Uf}^2=
\normc{f}^2,
$$
since both $F$ and $V$ are unitary.
Therefore, the family~$f_\beta$ is bounded.
Now, let $(\beta_n)$ be a sequence which converges to~$0$.
Let $f_n\eqd f_{\beta_n}$ and
$Q_n\eqd Q_{\beta_n}$. Since the sequence $(f_n)$ is bounded, we
can extract a weakly convergent subsequence $(f_{n_k})$. Let then
$\tilde{f}$ be the weak limit of this subsequence. On the one hand,
\begin{equation}
\label{TstarTftilde}
L^*Lf_{n_k}\rightharpoonup L^*L\tilde{f}
\quad\hbox{as}\quad
k\to\infty
\end{equation}
since $L^*L$ is bounded. On the other hand,
$$
Q_{n_k}^*Q_{n_k} f_{n_k}\rightharpoonup 0
\quad\hbox{as}\quad
k\to\infty,
$$
since $f_{n_k}$ is bounded and $Q_{n_k}^*Q_{n_k}$
converges pointwise to the null operator, so that
\begin{eqnarray*}
L^*Lf_{n_k}
&=&
(L^*L+Q_{n_k}^*Q_{n_k})f_{n_k}-Q_{n_k}^*Q_{n_k}f_{n_k}\\
&=&
L^*g-Q_{n_k}^*Q_{n_k}f_{n_k}\\
&=&
L^*Lf^\dagger-Q_{n_k}^*Q_{n_k}f_{n_k}\\
&\rightharpoonup&
L^*Lf^\dagger,
\end{eqnarray*}
as $k\to\infty$. Together with~\eqref{TstarTftilde}, this shows that
$L^*L\tilde{f}=L^*Lf^\dagger$, that is, by the injectivity of~$L$,
that $\tilde{f}=f^\dagger$.
It follows that the sequence $(f_n)$ converges weakly to~$f^\dagger$. 
Finally, by the weak lower semicontinuity of the norm,
$$
\norm{f^\dagger}\leq
\liminf_{n\to\infty}\norm{f_n}\leq
\limsup_{n\to\infty}\norm{f_n}\leq\norm{f^\dagger},
$$
which implies that $f_n\to f^\dagger$ as $n\to\infty$.
\endproof
\medskip

From now on, we shall make the following additional assumption on
the mollifier~$\varphi$:
\begin{assumption}\sf
\label{additional-assumption-phi-chapeau}
$\hat{\varphi}(\xi)$ decreases as $\mod{\xi}$ increases and,
for some positive~$s$,
\begin{equation}
\label{cond on conv kernel phi}
\mod{1-\sqrt{2\pi}\,\hat{\varphi}(\xi)}\sim\mod{\xi}^s
\quad\hbox{as}\quad
\xi \to 0.
\end{equation}
\end{assumption}

\begin{lemma}\sf
\label{lemma Alibaud et al}
Let $\varphi\in L^1(\Re)$ have unit integral and satisfy
Assumptions~\ref{assumption-phi-chapeau} and~\ref{additional-assumption-phi-chapeau}.
Let
\begin{equation}
\label{def m_beta and M_beta}
m_\beta = \min_{\vert\xi\vert=1} |1 - \sqrt{2\pi}\widehat{\varphi}(\beta \xi)|^2,\quad  \text{and} \quad  M_\beta = \max_{\vert\xi\vert=1} |1 -\sqrt{2\pi} \widehat{\varphi}(\beta \xi)|^2.
\end{equation}
Then the following hold:
\begin{itemize}
\item[(i)] $ 0< m_\beta \leq M_\beta \leq (1+ ||\varphi||_{L^1})^2$,

\item[(ii)] $M_\beta \to 0 \,\, \text{as}\,\, \beta \to 0$ and $\sup_{\beta \in (0,1]} \frac{M_\beta}{m_\beta} < \infty$,

\item[(iii)] there exist positive constants $\nu_0>0$ and $C_0>0$ such that, for all $\beta \in (0,1]$ and every $\xi \in \Re \setminus \{0\}$,
\begin{equation}
\label{key estimate alibaud}
\nu_0 \left(\vert \xi \vert^{2s} 1_{\{\vert \xi \vert \leq 1/\beta\}} + \frac{1}{M_\beta} 1_{\{\vert \xi \vert > 1/\beta\}} \right) \leq  \frac{|1-\sqrt{2\pi}\widehat{\varphi}(\beta \xi)|^2}{|1-\sqrt{2\pi}\widehat{\varphi}(\beta \xi/\vert \xi \vert)|^2} \leq C_0 \vert \xi \vert^{2s}.
\end{equation}
\end{itemize}
\end{lemma}
This lemma can be found together with its proof in \cite{alibaud2009variational}. 

\begin{lemma}\sf
\label{lemma completion lemma alibaud}
Consider the setting of Lemma \ref{lemma Alibaud et al} and let $p>0$.
Then the following hold:
\begin{itemize}
\item[(i)]
$
m_\beta \sim \beta^{2s} \quad \textrm{and} \quad M_\beta \sim \beta^{2s} \quad \textrm{as} \quad \beta \downarrow 0.
$
\item [(ii)] There exists a constant $C^{(1)}>0$ depending on $p$ such that
\begin{equation}
\label{upp bound norm f_beta}
\forall w \in H^p(\Re), \quad || (I - C_\beta)w||_{L^2}^2 \leq C^{(1)} \beta^{2(p \wedge s)} ||w||_{H^p}^2,
\end{equation}
\item[(iii)] There exists a constant $C^{(2)}>0$ depending on~$p$ such that
\begin{equation}
\label{upp bound (I-C_beta)^2 f}
\forall w \in H^{2p}(\Re), \quad || (I - C_\beta)^*(I - C_\beta)w||_{L^2}^2   \leq C^{(2)} \beta^{4(p \wedge s)} ||w||_{H^{2p}}^2,
\end{equation}
\end{itemize}
In~\eqref{upp bound norm f_beta}
and~\eqref{upp bound (I-C_beta)^2 f},
$p\wedge s = \min \{p,s\}$.
\end{lemma}
The proof of this Lemma is deferred to appendix.
With Lemmas~\ref{lemma Alibaud et al}
and~\ref{lemma completion lemma alibaud},
we are ready for the analysis of error estimates.

\section{Error estimates}
\label{section error estimate}

Henceforth, $g^\delta \in L^2(0,\infty)$ denotes
a noisy data satisfying the noise level condition
\begin{equation}
\label{noise level cond on data}
\norm{g-g^\delta}\leq\delta,
\end{equation}
and $f_\beta^\delta$ denotes the regularized solution
corresponding to noisy data $g^\delta$, that is,
\begin{equation}
\label{def reg sol noisy data}
f_\beta^\delta = \argmin_{f\in L^2(0,\infty)}
\norm{g^\delta - L f}^2 + \norm{ (I - C_\beta) V f}^2.
\end{equation} 

\begin{proposition}\sf
\label{Prop estimate data propagated error}
Consider the setting of Lemma~\ref{lemma Alibaud et al}.
Let $g^\delta$ be noisy data satisfying
\eqref{noise level cond on data}. 
Then there exists a constant $C >0$ independent
of~$\delta$ and~$\beta$ such that
\begin{equation}
\label{bound error prog term}
\norm{f_\beta - f_\beta^\delta} \leq C \frac{\delta}{\beta^{s}}
\quad \text{as} \quad
\beta \downarrow 0.
\end{equation}
\end{proposition}

\begin{proof}
Let $\beta>0$, we have
$f_\beta-f_\beta^\delta=(L^*L+Q_\beta^*Q_\beta)^{-1}L^*(g-g^\delta)$.
Since $\ran{L^*}=\ran{(L^*L)^{1/2}}$
(see e.g. \cite[Proposition 2.18]{engl1996regularization}), there exists
$\e^\delta\in L^2\segoo{0}{\infty}$ such that
$\tilde{f}\eqd L^*(g-g^\delta)=(L^*L)^{1/2}\e^\delta$.
We can see that $\norm{\e^\delta}=\norm{g-g^\delta}$.
As a matter of fact, denoting by~$P$ the orthogonal projection
onto the closure of the range of~$L$, we have 
\begin{eqnarray*}
\norm{\e^\delta}^2
&=&
\scal{\e^\delta}{\e^\delta}\\
&=&
\scal{(L^*L)^{-1/2}\tilde{f}}{(L^*L)^{-1/2}\tilde{f}}\\
&=&
\scal{\tilde{f}}{(L^*L)^{-1}\tilde{f}}\\
&=&
\scal{L^*(g-g^\delta)}{(L^*L)^{-1}L^*(g-g^\delta)}\\
&=&
\scal{g-g^\delta}{L(L^*L)^{-1}L^*(g-g^\delta)}\\
&=&
\scal{g-g^\delta}{P(g-g^\delta)}\\
&=&
\norm{g-g^\delta}^2,
\end{eqnarray*}
in which the sixth equality is stems from the fact that
$L(L^*L)^{-1}L^*=LL^\dagger=P$.
It follows that
$$
FV(f_\beta-f_\beta^\delta)= \left[\frac{(2\pi)^{1/4}\sqrt{\widehat{k}}}{\sqrt{2\pi}\widehat{k}+\mod{1-\sqrt{2\pi}\widehat\varphi_\beta}^2}\right] F V \varepsilon^\delta.
$$
Let $\Psi_\beta$ be the function defined by
\begin{equation*}
\label{def func Psi_beta moll deconv}
\Psi_\beta(\xi)=
\frac{(2\pi)^{1/4}\sqrt{\widehat{k}(\xi)}}{\sqrt{2\pi}\widehat{k}(\xi)+ \mod{1-\sqrt{2\pi}\widehat{\varphi}(\beta \xi)}^2}.
\end{equation*}
Using the unitarity of~$V$ and $F$, we see that
\begin{eqnarray}
\label{eq3}
\norm{f_\beta - f_\beta^\delta}
&=&
\int_{\Re}\Psi_\beta^2(\xi)
\modc{FV \varepsilon^\delta(\xi)}^2 \ud\xi \nonumber\\
&=&
\int_{|\xi|\leq r}
\Psi_\beta^2(\xi)\modc{FV \varepsilon^\delta(\xi)}^2\ud\xi +
\int_{|\xi| > r}
\Psi_\beta^2(\xi)\modc{FV \varepsilon^\delta(\xi)}^2\ud\xi
\end{eqnarray}
for any positive number~$r$.
Clearly, there exists $r_\circ>0$ such that
$$
\forall\xi\in\segcc{-r_\circ}{r_\circ},\quad
\sqrt{\widehat{k}(\xi)}>
\frac{1}{2}\sqrt{\widehat{k}(0)}.
$$
Therefore, if $|\xi| \leq r_\circ$, then
$$
\Psi_\beta(\xi) \leq (2\pi)^{-1/4}/\sqrt{\widehat{k}(\xi)} \leq
2(2\pi)^{-1/4}/\sqrt{\widehat{k}(0)},
$$
which implies that
\begin{equation}
\label{inte mol conv}
\int_{|\xi|\leq r_\circ}
\Psi_\beta^2(\xi)\modc{FV \varepsilon^\delta(\xi)}^2\ud\xi\leq
\frac{4}{\sqrt{2\pi}\widehat{k}(0)}
\int_{|\xi|\leq r_\circ}
\modc{FV \varepsilon^\delta(\xi)}^2\ud\xi.
\end{equation}
Now, define
$$
h_\circ(\beta)=\min
\setc{\modc{1-\sqrt{2\pi}\widehat{\varphi}(\beta\xi)}}
{\xi\in\Re,\;\mod{\xi}\geq r_\circ},
$$
in which the existence of the minimum is guaranteed by the Riemann-Lebesgue
lemma. Let $\xi_\circ$ be a point of attainment in the above minimum.
For every $\xi$ such that $\mod{\xi}\geq r_\circ$,
\begin{eqnarray*}
\Psi_\beta(\xi)\leq\frac{(2\pi)^{1/4}\sqrt{\widehat{k}(\xi)}}
{2(2\pi)^{1/4}\sqrt{\widehat{k}(\xi)}|1-\sqrt{2\pi} \widehat{\varphi}(\beta\xi)|}
\leq
\frac{1}{2\mod{1-\sqrt{2\pi}\widehat{\varphi}(\beta\xi_\circ)}}.
\end{eqnarray*}
This implies that 
\begin{equation}
\label{inte 2 mol deconv}
\int_{|\xi|>r_\circ}
\Psi_\beta^2(\xi)\modc{FV \varepsilon^\delta(\xi)}^2\ud\xi\leq
\frac{1}{4\mod{1-\sqrt{2\pi}\widehat{\varphi}(\beta\xi_\circ)}^2}
\int_{|\xi|>r}
\modc{FV \varepsilon^\delta(\xi)}^2\ud\xi.
\end{equation}
Applying \eqref{eq3} with $r =r_\circ$ together with \eqref{inte mol conv} and \eqref{inte 2 mol deconv} yields
\begin{equation}
\label{yyyy mol deconv}
\norm{f_\beta-f_\beta^\delta}^2\leq
\max\left\lbrace\frac{4}{\sqrt{2\pi}\widehat{k}(0)},
\frac{1}{4|1-\sqrt{2\pi}
\widehat{\varphi}(\beta \xi_\circ)|^2}\right\rbrace
\normc{FV\varepsilon^\delta}^2.
\end{equation}
Since $\mod{1-\sqrt{2\pi}\widehat{\varphi}(\beta \xi_\circ)}\to 0$
as $\beta$ goes to $0$, then for $\beta \ll 1$,
$\sqrt{2\pi}\widehat{k}(0)/4>4\mod{1-\sqrt{2\pi}
\widehat{\varphi}(\beta\xi_\circ)}^2$.
Hence~\eqref{yyyy mol deconv} implies that
\begin{equation}
\label{zzz mol deconv}
\norm{f_\beta-f_\beta^\delta}^2\leq
\frac{\norm{FV \varepsilon^\delta}^2}
    {2\mod{1-\sqrt{2\pi}\widehat{\varphi}(\beta \xi_\circ)}}
=\frac{\norm{\varepsilon^\delta}^2}{ 2\mod{1-\sqrt{2\pi}\widehat{\varphi}(\beta \xi_\circ)}}
=\frac{\norm{g - g^\delta }^2}
    {2\mod{1-\sqrt{2\pi}
    \widehat{\varphi}(\beta \xi_\circ)}}
\quad\text{as}\quad
\beta\downarrow 0. 
\end{equation}
The estimate \eqref{bound error prog term} follows
immediately by applying~\eqref{cond on conv kernel phi}
and~\eqref{noise level cond on data} to~\eqref{zzz mol deconv}.
\end{proof}

\begin{remark}\rm
From Theorem \ref{Theo consistency} and Proposition
\ref{Prop estimate data propagated error},
we deduce that for $\beta(\delta) = \delta^{\theta/s}$
with $\theta\in\segoo{0}{1}$, we have
\begin{equation}
\label{conver method noisy case}
\norm{f^\dagger-f_{\beta(\delta)}^\delta}\to 0
\quad\text{as}\quad
\delta\to 0.
\end{equation}
\end{remark}

Let us now study the regularization error $f^\dagger-f_\beta$.
It is well known that without imposing smoothness condition
on the exact solution~$f$ (or on the exact data $g$),
the regularization error of any regularization method
converges arbitrarily slowly to $0$
(see, e.g. \cite{schock1985approximate}).
Henceforth, we consider the smoothness condition 
\begin{equation}
\label{smoothness cond on f}
V f^\dagger \in H^p(\Re), \quad \text{with} \quad  p>0.
\end{equation}

\begin{proposition}\sf
\label{Prop characterization smoothness cond on f}
For $p\in\Ne$, the smoothness condition \eqref{smoothness cond on f} is equivalent to 
\begin{equation}
    \label{equivalent smoothness condition}
    \forall \,\ k =1,2,...p, \quad x^k (f^\dagger)^{(k)}(x) \in L^2(0,\infty).
\end{equation}
\end{proposition}

\begin{proof}
By mere computation, one gets that
\begin{equation}
    \label{eee 1}
(V f^\dagger )^{(p)}(x) = \sum_{k=0}^p a_{k,p} e^{(k+1/2)x}(f^\dagger)^{(k)}(e^x)
\end{equation}
where
$a_{0,p} = 1/2^p$, $a_{p,p} =1$ and $(a_{k,p})_{0<k<p}$ can be computed via the recurrence  equation
$$
a_{k,j+1} = (k + 1/2)a_{k,j} + a_{k-1,j}, \quad  k =1,..., j
$$
where the index $j$ runs from $1$ to $p-1$.
Next, the characterization \eqref{equivalent smoothness condition}
follows from \eqref{eee 1} and the fact that
$e^{(k+1/2)x}(f^\dagger )^{(k)}(e^x) \in L^2(\Re)$
if and only if $x^k (f^\dagger)^{(k)}(x) \in L^2(0,\infty)$ (using the change of variable $v = e^x$).
\end{proof}

Before getting into the analysis of convergence rates,
let us show that the smoothness \eqref{smoothness cond on f}
is nothing but a logarithmic source conditions generally
occurring in the regularization of exponentially ill-posed problems.

Henceforth, $\bar{L}:=L/\sqrt{\pi e}$ denotes the normalized
operator $L$, so that, $|||\bar{L}^*\bar{L}|||\leq e^{-1}$, normalization is necessary for logarithmic source condition (see, e.g. \cite{hohage2000regularization}).

\begin{proposition}\sf
\label{Prop 1 smoothness condition}
There exists two constants~$m$ and~$M$ such that
the smoothness condition~\eqref{smoothness cond on f}
is equivalent to the logarithmic source condition 
\begin{equation}
\label{log sour cond}
f^\dagger=f_p(\bar{L}^*\bar{L}) w,\quad
w\in L^2(0,\infty)
\end{equation} 
 satisfying 
\begin{equation}
\label{bound w and f}
m \norm{w} \leq \norm{Vf^\dagger }_{H^p} \leq M \norm{w}.
\end{equation}
where $f_{p}$ is defined as
\begin{equation*}
\forall \lambda \in (0,1), \quad f_{p}(\lambda) = (-\ln(\lambda))^{-p}.
\end{equation*}
\end{proposition}

\begin{proof}
Let $\xi \in \Re$, we have
$$
(1+|\xi|^2)^{p/2} f_p \left(\sqrt{2\pi} \widehat{k}(\xi)/\pi e \right) = \Phi(|\xi|) := \left[  \frac{\sqrt{1 + |\xi|^2}}{1 + \pi |\xi| + \ln\left( (1+e^{-2\pi |\xi|})/2 \right)}\right]^p
$$
Since the function $\Phi$ is strictly positive
and continuous on $[0,\infty)$ with
$\lim_{x\rightarrow\infty}\Phi(x)=1/\pi^p<\infty$,
we deduce that there exist constants~$m$ and~$M$
such that $\Phi(x)\in\segcc{m}{M}$ for all $x\geq 0$.
This implies that
\begin{equation}
\label{eq 1}
\forall\xi\in\Re,\quad
m\leq(1+\mod{\xi}^2)^{p/2} f_p
\left(\sqrt{2\pi}\widehat{k}(\xi)/\pi e \right) \leq M.
\end{equation}
Now, let the function $\varpi$ be formally defined
in the frequency domain as
\begin{equation}
\label{eq 2}
\widehat{\varpi}(\xi)=
\frac{1}{f_p \left(\sqrt{2\pi}\widehat{k}(\xi)/\pi e\right)}
\widehat{Vf}(\xi).
\end{equation}
From \eqref{eq 1}, we can readily see that the right
inequality implies that
$\varpi\in L^2(\Re)\Rightarrow Vf \in H^p(\Re)$
with $\norm{V f}_{H^p}\leq M\norm{\varpi}$
while the left inequality implies that
$Vf\in H^p(\Re)\Rightarrow\varpi\in L^2(\Re)$
with $m\norm{\varpi}\leq\norm{V f}_{H^p}$.
Finally, notice that \eqref{eq 2} can be rewritten as 
\begin{equation*}
FVf=
f_p\left(\sqrt{2\pi}\widehat{k}(\xi)/\pi e\right)F\varpi=
f_p\left(\sqrt{2\pi}\widehat{k}(\xi)/\pi e\right)FVV^*\varpi,
\end{equation*}
which implies that
\begin{equation}
\label{eq 22}
f=V^* F^*f_p\left(\sqrt{2\pi}\widehat{k}(\xi)/\pi e\right)FVw
\quad\text{with}\quad
w=V^*\varpi\in L^2(0,\infty).
\end{equation}
But \eqref{eq 22} is nothing but $f=f_p(\bar{L}^*\bar{L}) w$.
Finally notice that $\norm{w}=\norm{\varpi}$
since~$V$ is unitary.
\end{proof}

\begin{remark}\rm
From Proposition \ref{Prop 1 smoothness condition},
we can deduce that the smoothness
condition \eqref{smoothness cond on f}
is nothing but the logarithmic source
condition \eqref{log sour cond}.
Hence, letting $\rho:=\norm{w}$, we deduce
that the optimal-order convergence rate under
the noise level condition \eqref{noise level cond on data}
is  $C \rho f_{p}(\delta^2/\rho^2)$
with $C\geq 1$ independent of $\rho$ and $\delta$.
\end{remark}

\begin{proposition}\sf
\label{Prop rate reg error}
Consider the setting of Lemma \ref{lemma Alibaud et al}, and let
$g=Lf$. Assume that the unknown solution $f$ satisfies the smoothness condition \eqref{smoothness cond on f}. Then 
\begin{equation}
\label{estimate regul error}
\norm{ f - f_\beta } \leq \left(\frac{2s}{p\wedge 2s}\right)^p\rho f_p \left(\beta^{2s} \right) (1 + o(1)), \quad \text{as} \quad \beta \downarrow 0,
\end{equation}
where $\rho = \norm{w}$ and satisfies $ m \rho \leq \norm{V f}_{H^p} \leq M \rho$.
\end{proposition}

\begin{proof} We first remark that 
since $g=Lf$, and $L$ is injective, we have $f^\dagger =f$. Then
\begin{eqnarray}
\label{eq ii}
FV(f - f_\beta)(\xi)
&=&
\frac{\mod{1-\sqrt{2\pi}\widehat{\varphi}(\beta\xi)}^2}
    {\sqrt{2\pi}\widehat{k}(\xi)+\mod{1-\sqrt{2\pi}\widehat{\varphi}(\beta\xi)}^2} FVf(\xi)\nonumber\\
&=&
f_p \left(\sqrt{2\pi}\widehat{k}(\xi)/\pi e\right)
\frac{\mod{1-\sqrt{2\pi}\widehat{\varphi}(\beta\xi)}^2}
    {\sqrt{2\pi}\widehat{k}(\xi)+\mod{1-\sqrt{2\pi}\widehat{\varphi}(\beta\xi)}^2}
    FVw(\xi)\nonumber\\
&=&
f_p \left(\sqrt{2\pi}\widehat{k}(\xi)/\pi e\right)
FV\vartheta_\beta(\xi),
\end{eqnarray}
where 
\begin{equation}
\label{eq yy}
\vartheta_\beta=
V^* F^*\left(\frac{\mod{1-\sqrt{2\pi}\widehat{\varphi}(\beta \xi)}^2}{\sqrt{2\pi}
\widehat{k}(\xi)+\mod{1-\sqrt{2\pi}\widehat{\varphi}(\beta\xi)}^2}FV w\right).
\end{equation}
From \eqref{eq ii}, we deduce that
\begin{equation}
\label{eq hh}
FV(f-f_\beta)=
f_p\left(\sqrt{2\pi}\widehat{k}(\xi)/(\pi e)\right)FV\vartheta_\beta
\end{equation}
From~\eqref{eq hh}, \eqref{eq yy}, and the unitarity of~$F$ and~$V$, we get
\begin{equation}
\label{key estimate}
f-f_\beta=
f_p\left(\bar{L}^*\bar{L}\right)\vartheta_\beta, \quad \text{with} \quad \norm{\vartheta_\beta} \leq \norm{FV w} \leq \rho.
\end{equation}
Considering the  application of \cite[Proposition 1]{hohage2000regularization}
to the  particular case of logarithmic source function, and
observing that Theorem \ref{Theo consistency}, leads to
$\norm{\bar{L}(f-f_\beta)}\to 0$ as $\beta\downarrow 0$, we deduce that
\begin{equation}
\label{main1}
\norm{f-f_\beta}\leq\rho
f_p\left( \frac{\norm{ \bar{L}(f - f_\beta)}^2}{\rho^2} \right)( 1+o(1))
\quad\text{as}\quad
\beta\downarrow 0.
\end{equation} 
Using interpolation inequality, we have
\begin{eqnarray}
\label{eq9bis}
\norm{\bar{L}(f - f_\beta)}=\norm{(\bar{L}^*\bar{L})^{1/2}(f - f_\beta)}  
\leq\norm{(\bar{L}^*\bar{L})(f - f_\beta)}^{1/2}\norm{f - f_\beta}^{1/2}.
\end{eqnarray}
On the one hand, we have
\begin{eqnarray*}
\label{eq10}
\norm{(\bar{L}^*\bar{L})(f - f_\beta)}
&=&
\normc{\frac{\sqrt{2\pi}\widehat{k}(\xi)|1-\sqrt{2\pi}\hat{\varphi}(\beta\xi)|^2}{\sqrt{2\pi}\widehat{k}(\xi)+|1-\sqrt{2\pi}\hat{\varphi}(\beta\xi)|^2}\widehat{Vf}(\xi)}\\
&\leq &
\normc{|1-\sqrt{2\pi}\hat{\varphi}(\beta\xi)|^2\widehat{Vf}(\xi)}\\
&=&
\norm{(I-C_\beta)^*(I-C_\beta) Vf}\\
&\leq & \sqrt{C^{(2)}}\beta^{2(s\wedge(p/2))}\norm{Vf}_{H^p}\\
&\leq & \sqrt{C^{(2)}} M \rho \beta^{2(s\wedge(p/2))}
\end{eqnarray*}
in which we have used~\eqref{upp bound (I-C_beta)^2 f}
and~\eqref{smoothness cond on f}.
On the other hand, since $\norm{f-f_\beta}\to 0$ 
as $\beta \downarrow 0$, we have that for $\beta\ll 1$
\begin{equation}
\label{eq11}
\norm{f-f_\beta}\leq\frac{\rho}{\sqrt{C^{(2)}} M}.
\end{equation} 
Putting together \eqref{main1}, \eqref{eq9bis}, 
\eqref{eq11} and \eqref{Property log sourc function} yields \eqref{estimate regul error}.
\end{proof}

Now we are ready to state the main result of this section about
the order-optimality of our regularization method under the smoothness
condition~\eqref{smoothness cond on f}.

\begin{theorem}\sf
\label{Theorem order optim 1 under noisy data only}
Consider the setting of Lemma \ref{lemma Alibaud et al}.
Let~$f$ and~$g$ be in  $L^2(0,\infty)$ satisfying  $g=Lf$.
Let $g^\delta\in L^2(0,\infty)$ be a noisy data
satisfying \eqref{noise level cond on data}.
Assume that the solution~$f$ satisfies~\eqref{smoothness cond on f},
that is  $Vf\in H^p(\Re)$ for some $p>0$,
and let~$f_\beta^\delta$ be the reconstructed solution
defined by~\eqref{def reg sol noisy data}. Then for
the {\it a priori} selection rule
$\beta(\delta)=\left(\Theta_p^{-1}(\delta/\rho)\right)^{1/2s}$
with $\Theta_p(t)=\sqrt{t}f_p(t)$, we have
\begin{equation}
\label{order optim rata noisy data only}
\norm{f-f_{\beta(\delta)}^\delta }\leq K^{(1)} \rho f_p\left( \delta^2/\rho ^2 \right)(1 + o(1)) \quad \text{as} \quad \delta\to 0, 
\end{equation}
where $\rho = \norm{w}$ satisfies $C^{(1)} \rho \leq  \norm{Vf}_{H^p} \leq C^{(2)} \rho $ and $K^{(1)}$ is a constant independent of $\rho$ and $\delta$.
\end{theorem}

\begin{proof}
from Propositions \ref{Prop estimate data propagated error} and \ref{Prop 1 smoothness condition}, we get that there exist  positive constants $\bar{C}$, and $C$ such that
\begin{equation*}
\norm{f - f_{\beta(\delta)}^\delta } \leq \bar{C} \left(\frac{2s}{p\wedge 2s}\right)^p\rho f_p \left(\beta^{2s} \right)  + C \frac{\delta}{\beta^s},
\end{equation*}
For $\beta(\delta) = \left( \Theta_p^{-1}(\delta/\rho) \right)^{1/2s} $, we deduce that
\begin{equation}
    \label{eq eq}
\norm{f - f_{\beta(\delta)}^\delta } \leq \left( \bar{C }\left(\frac{2s}{p\wedge 2s}\right)^p + C \right) \rho f_p \left(\Theta_p^{-1}(\delta/\rho) \right) 
\end{equation}
The estimate (\ref{order optim rata noisy data only}) follows readily from the fact that $\rho f_p( \Theta_p^{-1}(\delta/\rho) )$ is nothing but the optimal rate under the assumption \eqref{log sour cond} (see, e.g. \cite[Theorem 1]{mathe2003geometry} and \cite[Theorem 2.1]{tautenhahn1998optimality}).
\end{proof}

\begin{corollary}\sf
\label{Corallary key}
Consider the setting of Theorem \ref{Theorem order optim 1 under noisy data only}.
For the {\it a priori} selection rule
$\beta(\delta) = c \delta^{\theta/s}$ with $\theta \in (0,1)$, we have
\begin{equation}
\label{rate sel rule}
\norm{f - f_{\beta(\delta)}^\delta } \leq K \rho f_p\left( \delta^{2 \theta} \right)(1 + o(1)) + C \delta^{1-\theta} = \mathcal{O}\left( f_p(\delta^2) \right) \quad \text{as} \quad \delta \to 0, 
\end{equation}
\end{corollary}

\begin{remark}\rm
We can see that the convergence rate
in \eqref{order optim rata noisy data only} is actually
order-optimal under the logarithmic source condition~\eqref{log sour cond}.
Hence Theorem~\ref{Theorem order optim 1 under noisy data only} implies
that the regularization method is order-optimal under the smoothness
condition~\eqref{smoothness cond on f}. Moreover from Corollary~\ref{Corallary key},
we have an {\it a priori} selection rule independent of the smoothness
of~$f$ leading to order-optimal rates (with respect to $\delta$) under
the smoothness condition~\eqref{smoothness cond on f}.
\end{remark}

\section{Numerical experiments}{\label{Section numerique}}

Given a function $f \in L_c^2(0,\infty)$, we set $g$ the restriction of the Laplace transform of $f$ on $[c,\infty)$, i.e.
\begin{equation}
 \forall x \geq c, \quad    g(x) = \int_0^{\infty} e^{-x t} f(t) \mathrm{d}t.
\end{equation}
We aim to approximate $f$ from $g$. 
In this section, we use the following approach based on partial differential equation for approximating the inverse Laplace transform $f$ of the function $g$. 

\subsection{Numerical approach}
Let $u$ be the real part of $L  f(z)$ for $z = x + i y \in \mathbb{C}$, that is,
\begin{equation}
\label{def function u}
    u(x,y) := \int_0^{\infty} e^{-x t} \cos(y t) f(t) \mathrm{d}t.
\end{equation}
Given that $L  f$ is holomorphic on $\Pi_c$, then $u$ is a harmonic function on $\Omega_c := \{(x,y) \in \mathbb{R}^2, \,\, x\geq c \}$. Moreover, $u(x,0) = g$ and using theorem of derivation under integral sign, one can check that $u_y(x,0) = 0$. Therefore, $u$ is nothing but the solution of the Cauchy problem
\begin{equation}
\label{cauchy problem u}
\begin{cases}
\Delta u(u,y) = 0, & x \geq c, \quad y \in \mathbb{R}, \\
\partial_y u(x,0)=0 , & x \geq c, \\
u(x,0)=g(x), & x \geq c.
\end{cases}  
\end{equation}
On the other hand, from \eqref{def function u}, we can see that $u(x,\cdot)$ is nothing but the Fourier cosine transform of the function $f(t)e^{-xt}$. Hence, if $u$ is known,  we can recover $f$ from $u$ via the equation
\begin{equation}
    \label{def f function of u}
\forall t \geq 0 \quad f(t) = \frac{2 e^{xt}}{\pi}\int_{0}^{\infty} \cos(yt) u(x,y) \mathrm{d}y = \frac{2 \sqrt{2}e^{xt}}{\sqrt{\pi}} \mathcal{R}eal \left(\mathcal{F} (u(x,\cdot) \right).
\end{equation}
The strategy used in this section for approximating $f$ consists in approximating the solution $u$ of the Cauchy problem \eqref{cauchy problem u} and then recover $f$ from \eqref{def f function of u} by applying Fourier transform to $u(x,\cdot)$ for a chosen $x \geq c$.

Notice that the Cauchy problem \eqref{cauchy problem u} is actually ill-posed. For regularizing equation \eqref{cauchy problem u}, we consider a variational formulation of mollification where roughly speaking, $u$ is approximated by $u_\beta$ where
\begin{equation}
    \label{def u beta num sec}
  u_\beta = \mathrm{argmin}_{u \in L^2(\Omega_c)}  \quad || A u  - g||_{L^2(c,\infty)}^2 + || (I  - C_\beta)u|_{\Omega_c}||_{L^2(\mathbb R^2)}^2
\end{equation}
where $C_\beta$ is the mollifier operator and $A$ is the operator the solution $u$ of \eqref{cauchy problem u} to the data $g$ (see for instance \cite{marechal2023mollifier}).

\subsection{Discretization of the Cauchy problem}
For the discretization of the Cauchy problem~\eqref{cauchy problem u}, we used fourth-order compact finite difference scheme method \cite{lele1992compact}. Consider the rectangular domain $[c,a] \times [0, b]$. We define a uniform grid $\Gamma$ on the bounded domain~$[c,L_x] \times [0, L_y]$:
$$
\Gamma_i^j = (x_i,y_j) \quad \text{with} \quad 
\begin{cases}
x_i = c + i h_x,   & i = 0,...,n_x \\
y_j = j h_y,   & j=0,...,n_y,
\end{cases}
$$
where $h_x$ and $h_y$ are the discretization steps given by
$h_x = (L_x-c)/n_x, \quad h_y = L_y/n_y$. \\
We recall here a four order compact finite difference scheme for approximating second derivative of a non-periodic function \cite{mehra2017suite}:
\begin{equation}
    \label{CDF 1 }
\frac{1}{10} y_{i+1}'' + y_{i}''+ \frac{1}{10} y_{i-1}''  =  \frac{6}{5} \frac{y_{i+1} - 2 y_{i} + y_{i-1}}{h^2},  \quad i =1,...,N-1  
\end{equation}
with boundary formulation
\begin{equation}
\label{CDF 2}
\begin{cases}
y_0'' + 10 y_1'' &= \frac{1}{h^2} \left( \frac{145}{12} y_0 - \frac{76}{3}y_1 + \frac{29}{2} y_2 - \frac{4}{3} y_3 + \frac{1}{12} y_4 \right) \\
y_N'' + 10 y_{N-1}'' &= \frac{1}{h^2} \left( \frac{145}{12} y_N - \frac{76}{3}y_{N-1} + \frac{29}{2} y_{N-2} - \frac{4}{3} y_{N-3} + \frac{1}{12} y_{N-4} \right).
\end{cases}
\end{equation}
We reformulate \eqref{CDF 1 } and \eqref{CDF 2} in the matrix formulation
\begin{equation}
\label{Matrix form CDF}
H_1 Y'' = \frac{1}{h^2} H_2 Y, \quad  Y = (y_0,y_1,...,y_{N-1},y_N)^\top.
\end{equation}
where $H_1 = I + R$ where $I$ denotes the identity matrix and $R$ is the matrix having zeros everywhere except at position $(1,2)$ and $(N+1,N)$ where both entries are equal to $10$. The matrix $H_2$ is the nearly tri-diagonal matrix with $-12/5$ on the diagonal, $6/5$ on the lower and upper diagonal, with the first row equal $[\frac{145}{12},- \frac{76}{3}, \frac{29}{2},- \frac{4}{3}, \frac{1}{12}, 0, \cdots, 0]$ and the last row being $[0, \cdots,0,\frac{1}{12},- \frac{4}{3} ,\frac{29}{2}, - \frac{76}{3}, \frac{145}{12}]$.

Now, denotes by $u_i^j$ an approximation of $u(x_i,y_j)$, and $U^j = (u_0^j, u_1^j,u_2^j, \cdots, u_{n_x}^j)^\top$. By applying the compact finite difference scheme \eqref{CDF 1 } and \eqref{Matrix form CDF} to approximate $u_{yy}$ and $u_{xx}$ respectively, we get that for $j=1,...,n_y-1$
\begin{eqnarray}
\label{Scheme CDF applied to Cauchy PB}
\frac{6}{5}\left( U^{j+1} - 2 U^{j} + U^{j-1} \right)  &=&
 h_y^2 \,\partial_{yy} \left( \frac{1}{10} U^{j+1} +  U^{j}+ \frac{1}{10}  U^{j-1} \right)  \nonumber \\
 &=&  - h_y^2  \, \partial_{xx} \left( \frac{1}{10}  U^{j+1} +  U^{j}+ \frac{1}{10}  U^{j-1} \right) \nonumber \\
 &=& - \frac{h_y^2}{h_x^2} H_1^{-1} H_2 \left[ \frac{1}{10} U^{j+1} +  U^{j}+ \frac{1}{10} U^{j-1} \right] 
\end{eqnarray}
By multiplying both sides of \eqref{Scheme CDF applied to Cauchy PB} by $H_1$ and denoting $r = \left(h_y/h_x\right)^2$, we get following iterative scheme 
\begin{equation}
    \label{iter scheme appr U}
   \left[ H_1 + \frac{r}{12} H_2 \right]  U^{j+1} = \left[ 2 H_1 - \frac{5}{6} r H_2 \right]  U^{j} - \left[ H_1 + \frac{r}{12} H_2 \right] U^{j-1}, \quad j=1,...,n_y-1.
\end{equation}
Given that $u(x,0) = g(x)$, then $U^0 = G := g(x_0,x_1,\cdots,x_{n_x})^\top$. On the other hand, by taking $j=0$ in \eqref{iter scheme appr U} and using the fact that $U^{-1} = U^{1}$ since the function $u$ by its definition is even along $y$-direction, we deduce that 
\begin{equation}
    \label{iter for j =0}
    \left[ H_1 + \frac{r}{12} H_2 \right]  U^{1} = \left[ H_1 - \frac{5}{12} r H_2 \right]  U^{0}
\end{equation}
From \eqref{iter scheme appr U} and \eqref{iter for j =0}, we can deduce that the cauchy problem \eqref{cauchy problem u} can be discretized into the matrix formulation 
\begin{equation}
\label{matrix formulation}
    A U = B,
\end{equation}
where
$$
U= \begin{pmatrix}
U^0 \\
U^1\\
U^2 \\
\vdots\\
\vdots \\
U^{n_y} 
\end{pmatrix}, \,\,
B= \begin{pmatrix}
G \\
0\\
0 \\
\vdots \\
\vdots \\
0 
\end{pmatrix}, 
\,\,
A = 
\begin{pmatrix}
I & 0 &  \cdots& \cdots &\cdots & 0\\
\frac{1}{2} D_1 & D & 0 &  & &\vdots\\
D  & D_1 & D  & 0 && \vdots\\
0 & D  & D_1 & D & \ddots  & \vdots\\
\vdots  & \ddots &\ddots  & \ddots & \ddots  &0\\
0  & \cdots & 0  & D  & D_1 & D
\end{pmatrix}, \quad
\text{with} \quad
\begin{cases}
D_1 = \frac{5}{6} r H_2 -  2 H1 \\
D = H_1 + \frac{r}{12} H_2.
\end{cases}
$$
From \eqref{def u beta num sec}, we approximate $U$ by $U_\beta$ defined as 
\begin{equation}
    U_\beta = \argmin_{U \in \mathbb{R}^{N}} || A U - B||_2^2 + || (I - C_\beta)U||_2^2,
\end{equation}
where $N = (n_x+1)(n_y+1)$. From the first order optimality condition, we can compute $U_\beta$ as the solution of the linear equation
\begin{equation}
    \label{equation U beta}
    \left[ A^*A +  (I - C_\beta)^*(I - C_\beta)\right] U_\beta = A^*B.
\end{equation}
For a noisy data $g^\delta$, it suffices to replace $B$ by $B^\delta$, where $B^\delta$ is defined similarly to $B$ except that the first block row is $G^\delta = g^\delta(x_0,x_1,\cdots,x_{n_x})^\top$.

\subsection{Simulation setting}
For the mollifier operator $C_\beta$, we consider a Cauchy kernel, that is, the kernel $\varphi$ is defined as
\begin{equation*}
    \varphi(x,y) = \frac{1}{\pi^2(1 + x^2)(1 + y^2)} .
\end{equation*}

For the selection of the regularization parameter $\beta$, we use the Morozov principle: Given $\tau >1$ and a noise level $\delta$, we consider the parameter $\beta(\delta,g^\delta)$ defined by
\begin{equation}
\label{rule choice beta}
 \beta(\delta,g^\delta) = \sup \setc{\beta>0}{|| A U_{\beta(\delta,g^\delta)} - B^\delta || < \tau \delta}.
\end{equation}

For the computing of $\beta(\delta,g^\delta)$ solution of \eqref{rule choice beta}, we use the following algorithm with $\tau = 1.01$ and $q=0.98$.

\begin{algorithm}
\begin{center}
\begin{algorithmic}[1]
\State Set $\beta_0 \gg 1$ and $q \in (0,1)$
\State Set $\beta = \beta_0$ (initial guess)
\While{$ || A U_{\beta} - B^\delta ||> \tau \delta$}
\State $\beta= q \times \beta$
\EndWhile
\end{algorithmic}
\end{center}
\caption{}
\label{Algo beta r}
\end{algorithm}

For the simulation, we consider three examples:

\textbf{Example 1}: $c =  0, \quad f(t) = \frac{1}{b} \sin(b t) e^{-a t},   \quad g(s) = \frac{1}{(a+s)^2 + b^2}$, with $a=1/2$ and $b = \sqrt{3}/2$.

\textbf{Example 2}: $c =  0, \quad f(t) = e^{-|t-2|},   \quad g(s) = \begin{cases} 
\frac{\exp(-2s)}{s+1} + \frac{\exp(-2) - \exp(-2s)}{s-1} & if\,\,\, x \neq 1 \\
\frac{5}{2}\exp(-2) & if\,\,\, x =1 \\
\end{cases}, $

\textbf{Example 3}: $c =0.5, \quad f(t) = 
\begin{cases} 
0 & if\,\,\, x <1 \\
0.5 & if\,\,\, x =1 \\
1 & if\,\,\, x>1
\end{cases}, 
 \quad g(s) = \frac{\exp(-s)}{s} $.

In all the simulations, we consider the grid $[c,7]\times [0,4]$ with $h_x = 0.25$ and $h_y = 0.025$. We approximate $\tilde{f}(t) =  e^{-ct}f(t) \in L^2(0,\infty)$ by $\tilde{f}_{\beta(\delta,g^\delta)}$ reconstructed from $u_{\beta(\delta,g^\delta)}$ for the choice $x = 0$, that is
\begin{equation}
\label{def f beta delta}
\forall t>0, \quad    \tilde{f}_{\beta(\delta,g^\delta)}(t)= \frac{2 e^{-ct}}{\pi}\int_{0}^{\infty} \cos(yt) u_{\beta(\delta,g^\delta)}(0,y) \mathrm{d}y = \frac{2\sqrt{2} e^{-ct}}{\sqrt{\pi}} \mathcal{R}eal \left(\mathcal{F}(u(0,\cdot). \right)
\end{equation}
In \eqref{def f beta delta}, we approximate $\tilde{f}_{\beta(\delta,g^\delta)}$ with the numerical procedure from \cite{bailey1994fast} using fast Fourier transform (FFT) algorithm. The noisy data $G^\delta$ is generated as 
$$
G^\delta = G + \eta  \vartheta, 
$$
where $\vartheta$ is a $(n_x +1)$-column vector of zero mean vector drawn from standard normal distribution and $\eta$ is a parameter allowing to control the noise level added. For each reconstruction, the relative error $Rel\_err$ is computed as 
$$
Rel\_err(u) = \frac{\norm{U-U_{\beta(\delta,g^\delta)}}_2}{\norm{U}_2}  \qquad Rel\_err(f) = \frac{\norm{\tilde{f}-\tilde{f}_{\beta(\delta,g^\delta)}}_2}{\norm{\tilde{f}}_2} 
$$

On Figure \ref{Fig 1} and \ref{Fig 2}, we illustrate the reconstruction $u_{\beta(\delta,g^\delta)}$ and $\tilde{f}_{\beta(\delta,g^\delta)}$ for $0.001$ percent noise level for each example.

On Table \ref{Tab 1}, we illustrate the relative error for $0.001$, $0.01$ and $0.1$ percent noise level for each example.

Figure \ref{Fig illustration quality reconstr} illustrates the quality of reconstruction achievable by the regularization approach described above.

In order to confirm the logarithmic convergence rates of the selection rules of the reconstruction error, On Figure \ref{Fig illus log conv rate}, we plot $\ln(rel\_err)$ versus $\ln(-\ln(\delta))$ for various values of $\delta$. We recall that if $rel\_err = O{f_p(\delta^2)}$ as $\delta \to 0$, then the curve $(\ln(-\ln(\delta)),\ln(rel\_err))$ should exhibit a line shape with slope equal to $-p$.

From Figure \ref{Fig illus log conv rate}, we can see that the reconstruction error actually exhibit a logarithmic rate with numerical order decreasing from Example 1 to Example 3. This confirms the theoretical results, since the smoothness of the target function $f$ decreases from Example 1 to Example 3.

\begin{figure}[h!]
\begin{center}
\includegraphics[scale=0.65]{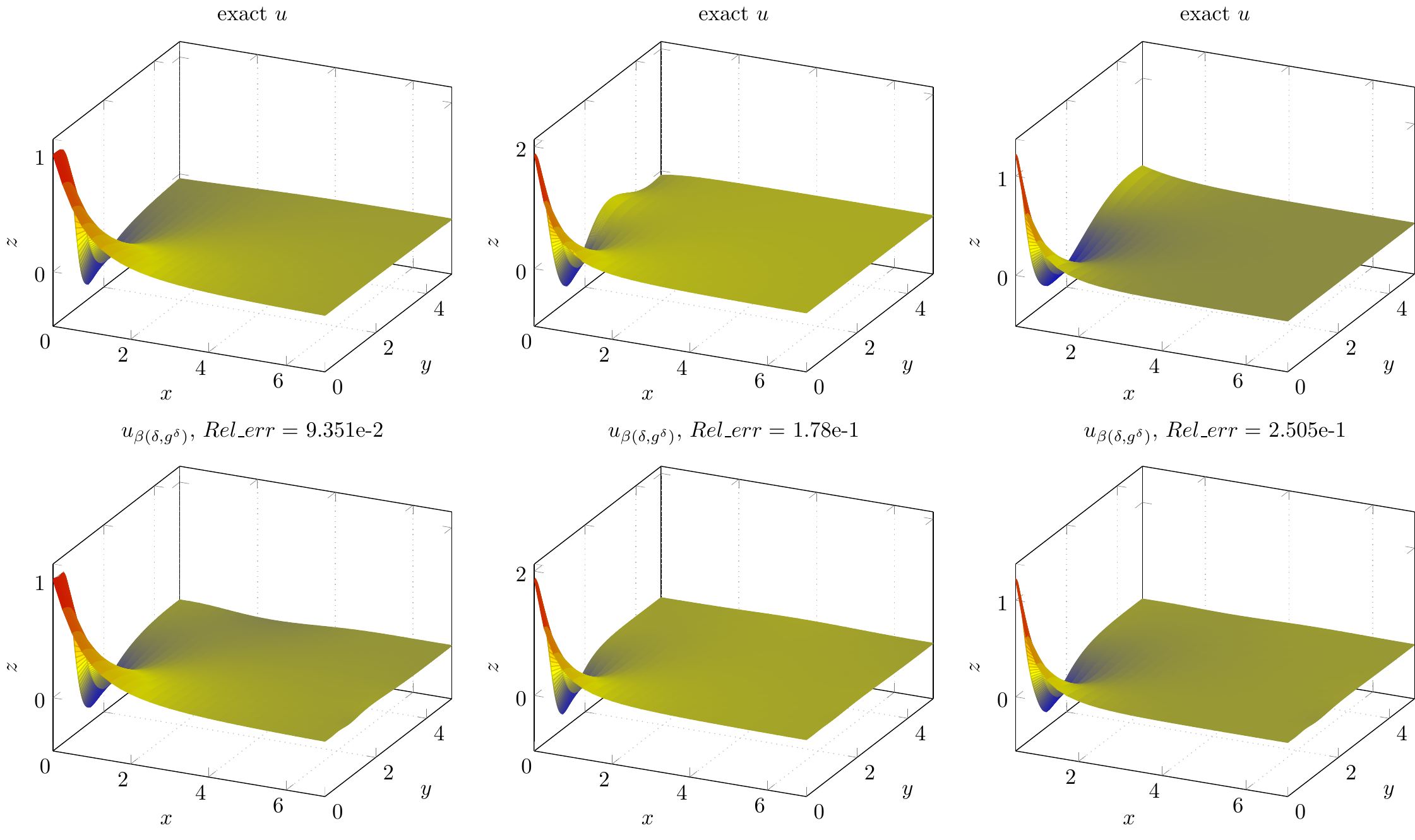} 
\end{center}
\caption{Illustration of reconstructed solution $u_{\beta(\delta,g^\delta)}$ for Example 1 (first column), 2 (second column) and 3 (third column)}
\label{Fig 1}
\end{figure}

\begin{figure}[h!]
\begin{center}
\includegraphics[scale=0.65]{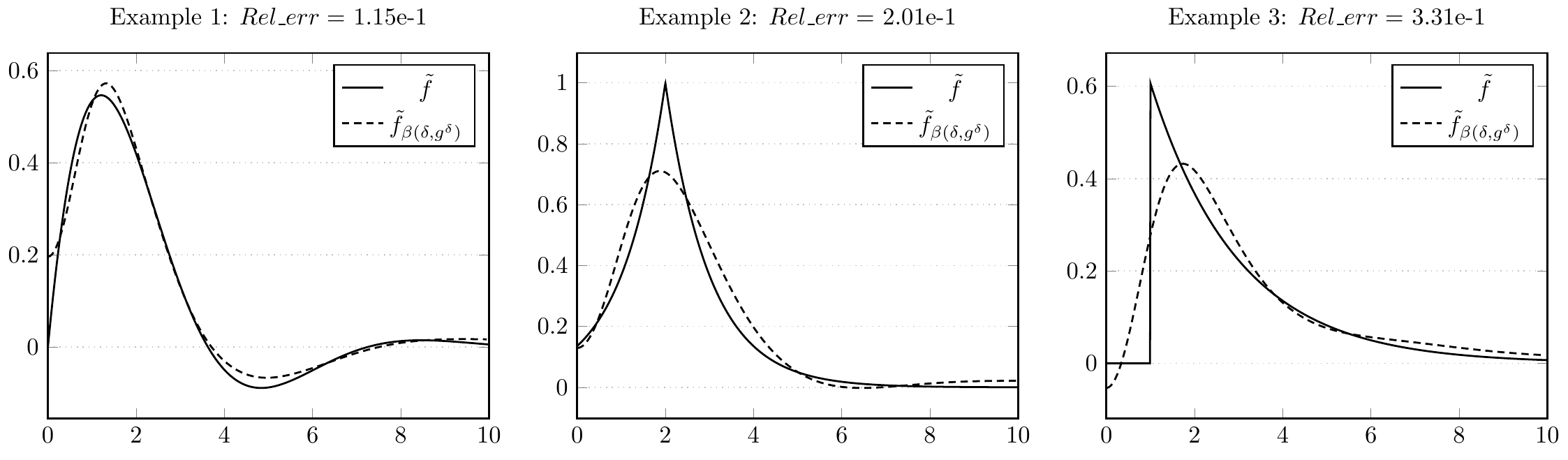} 
\end{center}
\caption{Illustration of reconstructed solution of $\tilde{f}_{\beta(\delta,g^\delta)}$ for Examples 1, 2 and 3}
\label{Fig 2}
\end{figure}

\begin{table}[]
    \centering
\includegraphics[scale=1]{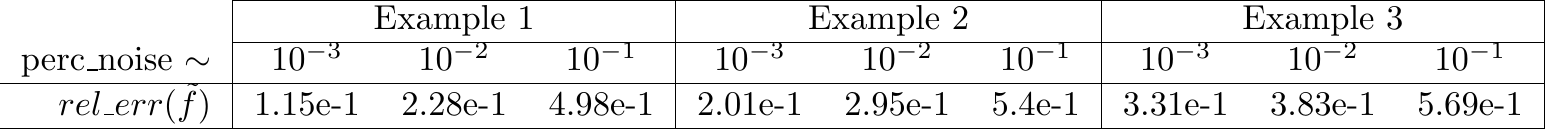} 
    \caption{Relative reconstruction error of $\tilde{f}$ for  $0.001$, $0.01$ and $0.1$ percent noise level}
    \label{Tab 1}
\end{table}

\begin{figure}[]
    \centering
\includegraphics[scale=1]{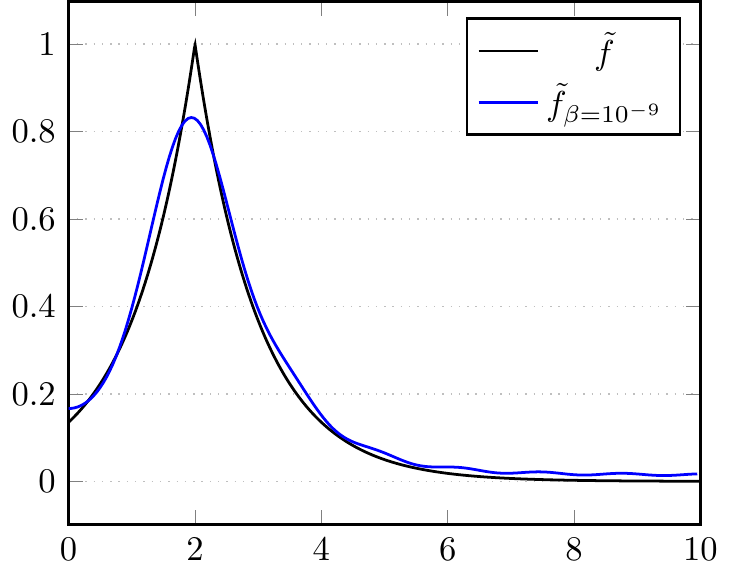} 
    \caption{Illustration approximation of $\tilde{f}$ for  $0.001$ percent noise level in Example 2 for $\beta  = 1e-9$.}
    \label{Fig illustration quality reconstr}
\end{figure}

\begin{figure}[]
    \centering
\includegraphics[scale=0.65]{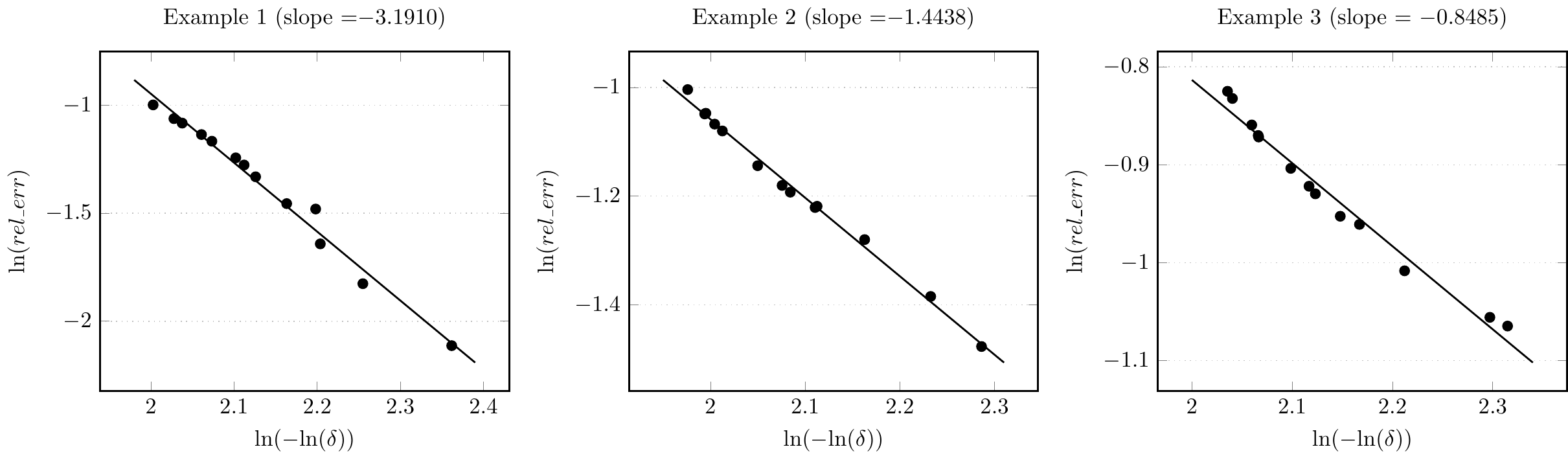} 
    \caption{Illustration numerical logarithmic convergence rates.}
    \label{Fig illus log conv rate}
\end{figure}

\appendix

\section{Unique continuation of Holomorphic functions}

Let $F(z)$ be a holomorphic function on
$\Pi_c=\set{z\in\Ce}{\re(z)>c}$, satisfying 
\bean
\label{fBound}
\left |F(z)\right| \leq M,
\quad \forall z\in \Pi_c.
\eean

\begin{proposition}\cite{ghanmi2022recovering}\rm
Denote $w\pm (z)=\frac{2}{\pi} (\frac{\pi}{2}\mp  arg(z-c))$,
the harmonic measures of the half line
$\set{z\in \Ce}{\re(z)>c,\, \im(z)=\{0\}}$ in 
$\segoo{c}{\infty}\times\Re_\pm$. They are the unique
solution to the systems
\bea
\Delta w^\pm (z) &= 0, \quad z\in \segoo{c}{\infty}\times \Re_\pm, \\
w(z) &=1, \quad \im(z)=0, \\
w(z) &=0, \quad { \re(z)= c}.
\eea
\end{proposition}

\begin{theorem}\rm
\label{thmUC}
The function $F(z)$ satisfies
\begin{equation}
|F(z)| \leq  M \left(\frac{\ep}{M}\right)^{w^\pm(z)}, \quad 
\forall z\in\segoo{c}{\infty} \times \Re_\pm,
\end{equation}
where $\ep=\norm{F(z)}_{L^\infty(\segoo{c}{\infty}\times \{0\})}$.
\end{theorem}

\begin{proof}
Recall that $F$ is holomorphic and satisfies \eqref{fBound}.
The holomorphic unique continuation of the function $z\rightarrow F(z)$
using the two constant Theorem \cite{nevanlinna1970analytic}, gives 
\bea
|F(z)| \leq M^{1-w^\pm(z)} |F(z)|^{w^\pm(z)}_{L^\infty(]c, \infty[ \times \{0\} )} \leq M \left(\frac{\ep}{M}\right)^{w^\pm(z)}, \quad 
\forall z\in  ]c,\infty[ \times \R_\pm,
\eea
where $\ep=\|F(z)\|_{L^\infty(]c, \infty[ \times \{0\} )}$.
\end{proof}

\section{Proof of Lemma~\ref{lemma completion lemma alibaud}}

\begin{enumerate}
\item[(i)]
This follows readily from \eqref{cond on conv kernel phi}.
\item[(ii)]
Let $p \in \Re$ and $w \in H^p(\Re)$,
if $p\leq s$,
\begin{eqnarray*}
\norm{(I - C_\beta)w}_{L^2}^2
&=&
\int_{\Re} |1-\widehat{\varphi}(\beta \xi)|^{2(1-p/s)}\,  |1-\widehat{\varphi}(\beta \xi)|^{2p/s} |\widehat{w}(\xi)|^2 \ud \xi \quad \text{by Parseval identity}\\
& \leq &
(1+\norm{\varphi}_{L^1(\Re)})^{2(1-p/s)} \int_{\Re} |1-\widehat{\varphi}(\beta \xi)|^{2p/s} |\widehat{w}(\xi)|^2 \ud \xi\\
&=&
(1+\norm{\varphi}_{L^1(\Re)})^{2(1-p/s)}\int_{\Re} |1-\widehat{\varphi}(\beta \xi/\vert \xi \vert)|^{2p/s} \mod{\frac{|1-\widehat{\varphi}(\beta \xi)|^2}{|1-\widehat{\varphi}(\beta \xi/\vert \xi \vert)|^{2}} }^{p/s} |\widehat{w}(\xi)|^2 \ud \xi \\
&\leq &
(1+\norm{\varphi}_{L^1(\Re)})^{2(1-p/s)} C_0^{p/s} M_\beta^{p/s}  \int_{\Re} \vert \xi \vert^{2p} |\widehat{w}(\xi)|^2 \ud \xi \quad \text{from} \,\, \eqref{def m_beta and M_beta} \,\, \text{and} \,\, \eqref{key estimate alibaud} \\
&\leq &
(1+\norm{\varphi}_{L^1(\Re)})^{2(1-p/s)} C_0^{p/s} M_\beta^{p/s}  \int_{\Re} (1+\vert \xi \vert^{2})^p |\widehat{w}(\xi)|^2 \ud \xi \\
& \leq &
C^{(1)} \beta^{2p}\norm{f}_{H^p}^2 \quad \text{using (ii)}.
\end{eqnarray*}
For $p>s$,
\begin{eqnarray*}
\norm{(I - C_\beta)w}_{L^2}^2
&=&
\int_{\Re} |1-\widehat{\varphi}(\beta \xi/\vert \xi \vert)|^2 \frac{|1-\widehat{\varphi}(\beta \xi)|^2}{|1-\widehat{\varphi}(\beta \xi/\vert \xi \vert)|^2} |\widehat{w}(\xi)|^2 \ud \xi\\
&\leq &
M_\beta C_0 \int_{\Re} \vert \xi \vert^{2s} |\widehat{w}(\xi)|^2 \ud \xi
\quad \text{from} \,\, \eqref{def m_beta and M_beta} \,\, \text{and} \,\, \eqref{key estimate alibaud} \\
&\leq  & 
C^{(1)} \beta^{2s}\norm{w}_{H^s}^2 \quad \text{using}\quad (i).
\end{eqnarray*}
\item[(iii)]
Let $p\in\Re$ and $w \in H^{2p}(\Re)$, if $p\leq s$, 
\begin{eqnarray*}
\lefteqn{\norm{(I - C_\beta)^*(I - C_\beta)w}_{L^2}^2}\\
&=&
\int_{\Re} |1-\widehat{\varphi}(\beta \xi)|^{4(1-p/s)}\,  |1-\widehat{\varphi}(\beta \xi)|^{4p/s} |\widehat{w}(\xi)|^2 \ud\xi\\
&\leq &
\tilde{C} \int_{\Re} |1-\widehat{\varphi}(\beta \xi)|^{4p/s} |\widehat{w}(\xi)|^2 \ud \xi \quad \text{with} \quad \tilde{C} = (1 + \norm{\varphi}_{L^1(\Re)})^{4(1-p/s)}\\
&=&
\tilde{C} \int_{\Re} |1-\widehat{\varphi}(\beta \xi/\vert \xi \vert)|^{4p/s} \mod{\frac{|1-\widehat{\varphi}(\beta \xi)|^2}{|1-\widehat{\varphi}(\beta \xi/\vert \xi \vert)|^{2}} }^{2p/s} |\widehat{w}(\xi)|^2 \ud \xi \\
&\leq &
\tilde{C} C_0^{2p/s} M_\beta^{2p/s}  \int_{\Re} \vert \xi \vert^{4p} |\widehat{w}(\xi)|^2 \ud \xi \quad (\text{from} \,\, \eqref{def m_beta and M_beta} \,\, \text{and} \,\, \eqref{key estimate alibaud}) \\
&\leq &
C^{(2)} \beta^{4p}\norm{w}_{H^{2p}(\Re)}^2 \quad (\text{using (ii)}).
\end{eqnarray*}
For $p>s$, 
\begin{eqnarray*}
\label{eqxbis}
\norm{(I - C_\beta)^*(I - C_\beta)w}_{L^2}^2
&=&
\int_{\Re} |1-\widehat{\varphi}(\beta \xi/\vert \xi \vert)|^4 \mod{\frac{|1-\widehat{\varphi}(\beta \xi)|^2}{|1-\widehat{\varphi}(\beta \xi/\vert \xi \vert)|^2}}^2 |\widehat{w}(\xi)|^2\ud\xi\\
&\leq &
M_\beta^2 C_0^2 \int_{\Re} \mod{\xi}^{4s} |\widehat{w}(\xi)|^2\ud\xi
\quad(\text{from}\;\;
\eqref{def m_beta and M_beta}
\;\;\text{and}\;\;
\eqref{key estimate alibaud}) \\
&\leq &
C^{(2)} \beta^{4s}\norm{w}_{H^{2s}}^2
\quad (\text{using (i)}).
\end{eqnarray*}
\end{enumerate}

\section{A technical lemma}

The following lemma exhibits some estimates about
the logarithmic source function $f_q$ which is used repeatedly in the paper.
\begin{lemma}\sf
\label{Lemma prop log sour func}
Let $q>0$ and the function $f_q$ defined by
$ \forall t \in (0,1]$, $f_q(t) = (-\ln(t))^{-q}$.
Then for all $a,b >0$ we have
\begin{equation}
\label{Property log sourc function}
\begin{cases}
\vspace{0.1cm}
\text{if}\,\, \lambda \leq 1, & \forall t \in (0,1), \quad f_q(\lambda t^a) \leq \max \left\lbrace 1,\left(\frac{b}{a}\right)^q \right\rbrace  f_q(\lambda t^b)\\\vspace{0.1cm}
\text{if}\,\, \lambda >1,     & \forall t \in (0, \lambda^{-\frac{2}{a}}), \quad f_q(\lambda t^a) \leq \max \left\lbrace 1, \left( \frac{2b-a}{a}\right)^q \right\rbrace  f_q(\lambda t^b)
\end{cases}
\end{equation}
Moreover, 
\begin{equation}
\label{Property log sourc function 2}
\begin{cases}
\vspace{0.cm}
\text{if}\,\, \lambda \leq 1, & \forall t \in (0,1), \quad f_q(\lambda t) \leq   f_q(t)\\\vspace{0.cm}
\text{if}\,\, \lambda >1,     & \forall t \in (0, \lambda^{-2}), \quad f_q(\lambda t) \leq 2^q   f_q(t).
\end{cases}
\end{equation}
\end{lemma}

\begin{proof} We have
\begin{equation}
\label{eqref1}
\forall t \in (0,1), \quad \frac{f_q(\lambda t^a)}{f_q(\lambda t^b)}  = j(x) := \left( \frac{ 1 + a x}{1 + bx} \right)^{-q}, \quad \text{with} \quad x = \frac{\ln t}{\ln \lambda}.
\end{equation}
If $\lambda \in (0,1)$ then $x>0$. For $a \geq b$,  $j(x)$ is obviously less than $1$. For $a < b$ the function $j$ is increasing $j(x)$ is bounded by above on $\Re^+$ by $\lim_{x \to \infty} j(x) = (a/b)^{-q}$.\\
Now for $\lambda >1$, and $t \in (0, \lambda^{-\frac{2}{a}})$, we have $x \leq -2/a$. For $a \geq b$, then right hand side in \eqref{eqref1} is bounded above by $1$. For $a < b$, the function $j$ is increasing and thus bounded by above on $(-\infty,-2/a)$ by $j(-2/a) = \left( \frac{a}{2b-a}\right)^{-q}$.
The first inequality in \eqref{Property log sourc function 2} follows readily from the fact the function $f_q$ is increasing on $(0,1)$. For the second inequality, we have
\begin{equation}
\label{eqref2}
\forall t \in (0,1), \quad \frac{f_q(\lambda t)}{f_q(t)}  = k(x) := \left( 1 + y \right)^{-q}, \quad \text{with} \quad y = \frac{\ln \lambda}{\ln t}.
\end{equation}
For $t \leq \lambda^{-2}$, $y>-1/2$, and since the function $k$ in \eqref{eqref2} is decreasing, we deduce that for all $y>-1/2$, $k(y) \leq k(-1/2) = 2^q$ whence the second inequality in \eqref{Property log sourc function 2}.
\end{proof}

\bibliographystyle{abbrv}
\bibliography{laplace}

\end{document}